\documentclass{amsart}
\usepackage{amsfonts, amsbsy, amsmath, amssymb}

\newtheorem{thm}{Theorem}[section]
\newtheorem{lem}[thm]{Lemma}

\numberwithin{equation}{section}

\begin{document}

\title{A Class of Permutation Binomials over Finite Fields}

\author[Xiang-dong Hou]{Xiang-dong Hou*}
\address{Department of Mathematics and Statistics,
University of South Florida, Tampa, FL 33620}
\email{xhou@usf.edu}
\thanks{* Research partially supported by NSA Grant H98230-12-1-0245.}

\keywords{finite field, hypergeometric sum, permutation polynomial}

\subjclass[2000]{11T06, 11T55, 33C05}

\begin{abstract}
Let $q>2$ be a prime power and $f={\tt x}^{q-2}+t{\tt x}^{q^2-q-1}$, where $t\in\Bbb F_q^*$. It was recently conjectured that $f$ is a permutation polynomial of $\Bbb F_{q^2}$ if and only if one of the following holds: (i) $t=1$, $q\equiv 1\pmod 4$; (ii) $t=-3$, $q\equiv \pm1\pmod{12}$; (iii) $t=3$, $q\equiv -1\pmod 6$. We confirm this conjecture in the present paper.
\end{abstract}

\maketitle

\section{Introduction}

Let $q$ be a prime power and $\Bbb F_q$ the finite field with $q$ elements. A polynomial $f\in\Bbb F_q[{\tt x}]$ is called a {\em permutation polynomial} (PP) of $\Bbb F_q$ if the mapping $x\mapsto f(x)$ is a permutation of $\Bbb F_q$. Nontrivial PPs in simple algebraic forms are rare. Such PPs are sometimes the result of the mysterious interplay between the algebraic and combinatorial structures of the finite field. Permutation {\em binomials} over finite fields are particularly interesting for this reason, and they have attracted the attention of many researchers over decades; see 
\cite{Car62, MPW, Mas-Zie09, Nie-Rob82, Tur88, Wan87, Wan94}. In these references, the reader will find not only many interesting results on permutation binomials but also plenty challenges that remain.

The main result of the present paper is the following theorem:

\begin{thm}\label{T1.1}
Let $f={\tt x}^{q-2}+t{\tt x}^{q^2-q-1}\in\Bbb F_q[{\tt x}]$, where $q>2$ and $t\in\Bbb F_q^*$. Then $f$ is a PP of $\Bbb F_{q^2}$ if and only if one of the following occurs: 
\begin{itemize}
  \item [(i)] $t=1$, $q\equiv 1\pmod 4$; 
  \item [(ii)] $t=-3$, $q\equiv \pm1\pmod{12}$; 
  \item [(iii)] $t=3$, $q\equiv -1\pmod 6$.
\end{itemize} 
\end{thm}

In fact, Theorem~\ref{T1.1} was conjectured in \cite{FHL}; it originated from a recent study of certain permutation polynomials over finite fields defined by a functional equation. We will briefly describe this connection in Section 4.

The attempt to prove Theorem~\ref{T1.1} has led to the discovery of a curious hypergeometric identity stated in Theorem~\ref{T1.2}. In return, Theorem~\ref{T1.2} clears the last hurdle in the proof of Theorem~\ref{T1.1}. 

\begin{thm}\label{T1.2}
Let $n\ge 0$ be an integer. Then we have
\begin{equation}\label{1.1}
\begin{split}
&\sum_{k\le 2n+1}\binom{2n+1}k\biggl(\prod_{j=1}^{2n+1}(6n-2k+4-2j)\biggr)(-1)^k\,3^{2k+1}\cr
+&\sum_{k\le 2n+1}\binom{2n+1}k\biggl(\prod_{j=1}^{2n+1}(6n-2k+5-2j)\Biggr)(-1)^k\,3^{2k}=0.
\end{split}
\end{equation}
\end{thm}

The proofs of Theorems~\ref{T1.2} and \ref{T1.1} are given in Sections 2 and 3, respectively.

\medskip

\noindent{\bf Remark.} For $q>2$, we can write the polynomial $f$ as $f\equiv {\tt x}^{q^2-2}h({\tt x}^{q-1})\pmod{{\tt x}^{q^2}-{\tt x}}$, where $h({\tt x})={\tt x}+t{\tt x}^q$. According to \cite[Lemma 2.1]{Zie09}, $f$ is a PP of $\Bbb F_{q^2}$ if and only if ${\tt x}^{q^2-2}h({\tt x})^{q-1}$ permutes the $(q-1)$st powers in $\Bbb F_{q^2}^*$. This observation, though interesting in its own right, does not seem to be useful in our approach.

\section{Proof of Theorem~\ref{T1.2}}

Let 
\[
\begin{split}
F_1(n,k)=\,&\binom{2n+1}k\biggl(\prod_{j=1}^{2n+1}(6n-2k+4-2j)\biggr)(-1)^k\,3^{2k+1},\cr
F_2(n,k)=\,&\binom{2n+1}k\biggl(\prod_{j=1}^{2n+1}(6n-2k+5-2j)\biggr)(-1)^k\,3^{2k},\cr
S_1(n)=\,&\sum_kF_1(n,k),\cr
S_2(n)=\,&\sum_kF_2(n,k).
\end{split}
\]
Using Zeilberger's algorithm \cite[Ch.\;6]{A=B}, \cite{Pau-Sch95}, we find that
\[
\begin{split}
&F_1(n+2,k)+24(36n^2+126n+113)F_1(n+1,k)+46656(n+1)^2(2n+3)^2F_1(n,k)\cr
=\,&G_1(n,k+1)-G_1(n,k),
\end{split}
\]
where $G_1(n,k)=F_1(n,k)R_1(n,k)$,
\[
\begin{split}
&R_1(n,k)\cr
=\,&-\frac{32k(3n-k+2)}{(n-k+1)(n-k+2)\prod_{j=2}^5(2n-k+j)}\cdot\bigl(264240-321108k +142242k^2 \cr
&-27228k^3+1902k^4+1434774n-1559605kn +612100k^2n-102647k^3n \cr
&+6194 k^4n+3361281n^2-3199801kn^2 +1081204k^2n^2-152528k^3n^2 \cr
&+7484k^4n^2 +4437783n^3-3594830kn^3+1003340k^2n^3-111631k^3n^3\cr
&+3976k^4n^3 +3611829n^4-2388503kn^4+515900k^2n^4-40234k^3n^4\cr
&+784k^4n^4 +1855833n^5-938595kn^5+139350k^2n^5-5712k^3n^5 \cr
&+587970n^6 -201978kn^6+15444k^2n^6+105030n^7-18360kn^7+8100n^8\bigr).
\end{split}
\]
By the same algorithm, we have
\[
\begin{split}
&F_2(n+2,k)+24(36n^2+126n+113)F_2(n+1,k)+46656(n+1)^2(2n+3)^2F_2(n,k)\cr
=\,&G_2(n,k+1)-G_2(n,k),
\end{split}
\]
where $G_2(n,k)=F_2(n,k)R_2(n,k)$,
\[
\begin{split}
&R_2(n,k)\cr
=\,&-\frac{4k(6n-2k+5)}{(2n-2k+3)(2n-2k+5)\prod_{j=2}^5(2n-k+j)}\cdot\bigl(5518665-6111039k\cr
&+2516532k^2-455172k^3+30432k^4+29095596n-29034593kn +10674112k^2n \cr
&-1703836k^3n+99104 k^4n+66125967n^2-58228898kn^2 +18571132k^2n^2 \cr
&-2512456k^3n^2+119744k^4n^2+84611256n^3 -63891952kn^3+16960112k^2n^3 \cr
&-1823312k^3n^3+63616k^4n^3 +66666108n^4-41422240kn^4+8573312k^2n^4 \cr
&-650912k^3n^4+12544k^4n^4 +33120768n^5-15865680kn^5+2273856k^2n^5 \cr
&-91392k^3n^5 +10132560n^6-3323808kn^6+247104k^2n^6+1745280n^7 \cr
&-293760kn^7+129600n^8\bigr).
\end{split}
\]
Therefore, both $S_1(n)$ and $S_2(n)$ satisfy the same second order recurrence relation:
\begin{gather*}
S_1(n+2)+24(36n^2+126n+113)S_1(n+1)+46656(n+1)^2(2n+3)^2S_1(n)=0,\\
S_2(n+2)+24(36n^2+126n+113)S_2(n+1)+46656(n+1)^2(2n+3)^2S_2(n)=0.
\end{gather*}
It is easy to check that
\[
\begin{split}
&S_1(0)=6 =-S_2(0),\cr
&S_1(1)=-3312 =-S_2(1).
\end{split}
\]
Hence $S_1(n)=-S_2(n)$ for all $n\ge 0$, which completes the proof of Theorem~\ref{T1.2}.

\medskip

\noindent{\bf Remark.} The hypergeometric sums $S_1(n)$ and $S_2(n)$ cannot be expressed in closed forms (in the sense of \cite[Definition~8.1.1]{A=B}). This fact has been proved using Algorithm Hyper \cite[Ch.\;8]{A=B}.

\medskip

The identity in Theorem~\ref{T1.2} can be expressed in the standard notation of hypergeomtric series. For an integer $k\ge 0$ and an element $a$ in any ring, let $(a)_k=a(a+1)\cdots(a+k-1)$ denote the rising factorial. We have  
\[
\begin{split}
&S_1(n)\cr
=\,&\sum_k\binom{2n+1}k\biggl(\prod_{j=1}^{2n+1}\bigl(6n-2k+4-2j\bigr)\biggr)(-1)^k3^{2k+1}\cr
=\,&\sum_k\binom{2n+1}k\biggl(\prod_{j=1}^{2n+1}\bigl(6n-2(2n+1-k)+4-2j\bigr)\biggr)(-1)^{2n+1-k}3^{2(2n+1-k)+1}\cr
&\kern8cm (k\mapsto 2n+1-k)\cr
=\,&-2^{2n+1}\cdot 3^{4n+3}\sum_k\binom{2n+1}k\biggl(\prod_{j=1}^{2n+1}(k+n+1-j)\biggr)(-1)^k 3^{-2k}\cr
=\,&-2^{2n+1}\cdot 3^{4n+3}\sum_{n+1\le k\le 2n+1}\binom{2n+1}k\biggl(\prod_{j=1}^{2n+1}(k+n+1-j)\biggr)(-1)^k 3^{-2k}\cr
=\,&-2^{2n+1}\cdot 3^{4n+3}\sum_{k\ge 0}\binom{2n+1}{k+n+1}\biggl(\prod_{j=1}^{2n+1}(k+2n+2-j)\biggr)(-1)^{k+n+1} 3^{-2(k+n+1)}\cr
&\kern8cm (k\mapsto k+n+1)\cr
=\,&(-1)^n\cdot 2^{2n+1}\cdot 3^{2n+1}\sum_{k\ge 0}\binom{2n+1}{k+n+1}\biggl(\prod_{j=1}^{2n+1}(k+j)\biggr)(-1)^k 3^{-2k}\cr
\end{split}
\]
In the above,
\[
\binom{2n+1}{k+n+1}=\frac{(-1)^{k+n+1}(-2n-1)_{k+n+1}}{(1)_{k+n+1}}=\frac{(-1)^{k+n+1}\,(-2n-1)_{n+1}\,(-n)_k}{(n+1)!\,(n+2)_k},
\]
\[
\prod_{j=1}^{2n+1}(k+j)=\prod_{j=1}^{2n+1}\Bigl(\frac{(j+1)_k}{(j)_k}\cdot j\Bigr)=(2n+1)!\cdot\frac{(2n+2)_k}{(1)_k}.
\]
So
\[
\begin{split}
S_1(n)\,&= - 2^{2n+1}\cdot 3^{2n+1}(-2n-1)_{n+1}\frac{(2n+1)!}{(n+1)!}\sum_{k\ge 0}\frac{(-n)_k\,(2n+2)_k}{(n+2)_k}\cdot\frac{(3^{-2})^k}{k!}\cr
&=(-1)^n\cdot 2^{2n+1}\cdot 3^{2n+1}\cdot (n+1)_{n+1}\cdot (n+2)_n\cdot{}_2F_1\biggl[\begin{matrix} -n,\;2n+2\vspace{1mm}\cr n+2\end{matrix}\biggm| 3^{-2}\biggr].
\end{split}
\]
Similarly, 
\[
\begin{split}
S_2(n)\,&=-2^{2n+1}\cdot 3^{4n+2}\sum_k\binom{2n+1}k\biggl(\prod_{j=1}^{2n+1}\Bigl(k+n+\frac 32-j\Bigr)\biggr)(-1)^k 3^{-2k}\cr
&=-2^{2n+1}\cdot 3^{4n+2}\sum_k\binom{2n+1}k\biggl(\prod_{j=-n}^n\Bigl(k+\frac 12+j\Bigr)\biggr)(-1)^k 3^{-2k}.
\end{split}
\]
In the above,
\[
\binom{2n+1}k=\frac{(-1)^k\,(-2n-1)_k}{k!},
\]
\[
\prod_{j=-n}^n\Bigl(k+\frac 12+j\Bigr)=\prod_{j=-n}^n\Bigl[\frac{(j+\frac 32)_k}{(j+\frac 12)_k}\cdot\Bigl(j+\frac 12\Bigr)\Bigr]=\Bigl(-n+\frac 12\Bigr)_{2n+1}\cdot\frac{(n+\frac 32)_k}{(-n+\frac 12)_k}.
\]
So 
\[
\begin{split}
S_2(n)\,&=-2^{2n+1}\cdot 3^{4n+2}\cdot\Bigl(-n+\frac 12\Bigr)_{2n+1}\sum_k\frac{(-2n-1)_k\,(n+\frac 32)_k}{(-n+\frac 12)_k}\cdot\frac{(3^{-2})^k}{k!}\cr
&=-2^{2n+1}\cdot 3^{4n+2}\cdot\Bigl(-n+\frac 12\Bigr)_{2n+1}\cdot{}_2F_1\biggl[\begin{matrix} n+\frac 32,\; -2n-1 \vspace{1mm}\cr -n+\frac 12\end{matrix}\biggm| 3^{-2}\biggr].
\end{split}
\]
Therefore, Theorem~\ref{T1.2} can be stated as
\[
{}_2F_1\biggl[\begin{matrix} -n,\; 2n+2 \vspace{1mm}\cr n+2\end{matrix}\biggm| 3^{-2}\biggr]=(-1)^n\,3^{2n+1}\frac{(-n+\frac 12)_{2n+1}}{(n+1)_{n+1}\,(n+2)_n}\;{}_2F_1\biggl[\begin{matrix} n+\frac 32,\;
-2n-1 \vspace{1mm}\cr -n+\frac 12\end{matrix}\biggm| 3^{-2}\biggr].
\]


\section{Proof pf Theorem~\ref{T1.1}}

Let $\Bbb Q_p$ denote the field of $p$-adic numbers and $\Bbb Z_p$ the ring of $p$-adic integers.
For an integer $a\ge 0$ and an element $z\in\Bbb Q_p$, we define $\binom za=\frac{(z-a+1)_a}{(1)_a}$. If $z\in\Bbb Q$, we also define
\[
\binom za^{\!\!*}=
\begin{cases}\displaystyle \binom za&\text{if}\ z\in\Bbb Z,\vspace{1mm}\cr
0&\text{otherwise}.
\end{cases}
\]

\begin{lem}\label{L3.0}
Let $q$ be a power of a prime $p$ and $a$ an integer with $0\le a\le q-1$. Let $z_1,z_2\in\Bbb Z_p$ such that $z_1\equiv z_2\pmod q$. Then $\binom{z_1}a\equiv \binom{z_2}a\pmod p$.
\end{lem}

\begin{proof}
Write $z_1=z+qw$, where $z\in\Bbb Z$, $w\in\Bbb Z_p$. It suffices to show that $\binom{z_1}a\equiv \binom za\pmod p$. There exists a sequence $w_n\in\Bbb Z$ such that $w_n\to w$ as $n\to\infty$. We have
\[
\begin{split}
\binom{z+qw_n}a\,&=\text{the coefficient of ${\tt x}^a$ in $(1+{\tt x})^{z+qw_n}$}\cr
&\equiv\text{the coefficient of ${\tt x}^a$ in $(1+{\tt x})^z (1+{\tt x}^q)^{w_n}$} \pmod p\cr
&=\binom za\kern 6cm\text{(since $a<q$)}.
\end{split}
\]
Letting $n\to\infty$ in the above, we have $\binom{z_1}a\equiv \binom za\pmod p$.
\end{proof}

\begin{lem}\label{L3.1}
Let $f$ be as in Theorem~\ref{T1.1}. Let $0<\alpha+\beta q<q^2-1$, where $0\le \alpha,\beta\le q-1$. Then 
\begin{equation}\label{3.1}
\begin{split}
&\sum_{x\in\Bbb F_{q^2}^*}f(x)^{\alpha+\beta q}\cr
=\,&
\begin{cases}
0\kern6.9cm{\text if}\ \alpha+\beta\ne q-1, \vspace{4mm}\cr
\displaystyle -(-1)^{\frac{\alpha+q}2}t^{-\frac{3\alpha+q}2}\biggl[(-1)^{\frac{q+1}2}t^{\frac{q-1}2}\sum_i\binom\alpha i\binom{\frac{3\alpha-1}2-i}\alpha^{\!\!*} (-1)^i\,t^{2i+1} \vspace{2mm}\cr
\displaystyle +\sum_i\binom\alpha i\binom{\frac{3\alpha-1}2-i+\frac{q+1}2}\alpha^{\!\!*}  (-1)^i\,t^{2i}\biggr] \kern 1.1cm\text{if}\ \alpha+\beta=q-1.
\end{cases}
\end{split}
\end{equation}
\end{lem}

\begin{proof}
We have
\[
\begin{split}
&\sum_{x\in\Bbb F_{q^2}^*}f(x)^{\alpha+\beta q}\cr
=\,&\sum_{x\in\Bbb F_{q^2}^*}(x^{q-2}+tx^{-q})^{\alpha+\beta q}\cr
=\,&\sum_{x\in\Bbb F_{q^2}^*}t^{\alpha+\beta q}x^{-q(\alpha+\beta q)}(1+t^{-1}x^{2q-2})^{\alpha+\beta q}\cr
=\,&t^{\alpha+\beta}\sum_{x\in\Bbb F_{q^2}^*}x^{-q\alpha-\beta}(1+t^{-1}x^{2(q-1)})^\alpha (1+t^{-1}x^{2(1-q)})^\beta \cr
=\,&t^{\alpha+\beta}\sum_{x\in\Bbb F_{q^2}^*}x^{-q\alpha-\beta}\sum_{i,j}\binom \alpha i\binom\beta j t^{-(i+j)}x^{2(q-1)(i-j)}.
\end{split}
\]
This sum is $0$ unless $-q\alpha-\beta\equiv 0\pmod {q-1}$, i.e., $\alpha+\beta\equiv 0\pmod{q-1}$. 

Assume $\alpha+\beta\equiv 0\pmod{q-1}$. Since $0<\alpha+\beta q<q^2-1$, we must have $\alpha+\beta=q-1$. Then $-q\alpha-\beta=-q\alpha-(q-1-\alpha)=-(\alpha+1)(q-1)$. Thus we have 
\[
\begin{split}
&\sum_{x\in\Bbb F_{q^2}^*}f(x)^{\alpha+\beta q}\cr
=\,&\sum_{x\in\Bbb F_{q^2}^*}\sum_{i,j}\binom \alpha i\binom{q-1-\alpha}j t^{-(i+j)}x^{(q-1)(2(i-j)-\alpha-1)}\cr
=\,&-\sum_{2(i-j)\equiv\alpha+1\kern-2.5mm \pmod{q+1}}\binom \alpha i\binom{q-1-\alpha}j t^{-(i+j)}.
\end{split}
\]
When $0\le i\le\alpha$ and $0\le j\le q-1-\alpha$, we have
\[
\alpha+1-2(q+1)<2(i-j)<\alpha+1+q+1.
\]
So the condition $2(i-j)\equiv\alpha+1\pmod{q+1}$ is satisfied if and only if $2(i-j)=\alpha+1$ or $\alpha+1-(q+1)$. Hence
\begin{equation}\label{3.2}
\begin{split}
&\sum_{x\in\Bbb F_{q^2}^*}f(x)^{\alpha+\beta q}\cr
=\,&-\biggl(\sum_{i-j=\frac{\alpha+1}2}+\sum_{i-j=\frac{\alpha-q}2}\biggr)\binom \alpha i\binom{q-1-\alpha}j t^{-(i+j)}\cr
=\,&-\biggl(\sum_{i+j=\frac{\alpha-1}2}+\sum_{i+j=\frac{\alpha+q}2}\biggr)\binom \alpha i\binom{q-1-\alpha}j t^{-(\alpha-i+j)}\kern1.25cm (i\mapsto \alpha-i)\cr
=\,&-\biggl(\sum_{i+j=\frac{\alpha-1}2}+\sum_{i+j=\frac{\alpha+q}2}\biggr)\binom \alpha i\binom{j+\alpha}\alpha (-1)^j t^{-(\alpha-i+j)}\kern1cm \text{(Lemma~\ref{L3.0})}\cr
=\,&-\sum_i\binom \alpha i\binom{\frac{3\alpha-1}2-i}\alpha^{\!\!*}  (-1)^{\frac{\alpha-1}2-i}t^{-\frac{3\alpha-1}2+2i}\cr 
&-\sum_i\binom \alpha i\binom{\frac{3\alpha-1}2-i+\frac{q+1}2}\alpha^{\!\!*}  (-1)^{\frac{\alpha+q}2-i}t^{-\frac{3\alpha+q}2+2i}\cr
=\,&-(-1)^{\frac{\alpha+q}2}t^{-\frac{3\alpha+q}2}\biggl[(-1)^{\frac{q+1}2}t^{\frac{q-1}2}\sum_i\binom\alpha i\binom{\frac{3\alpha-1}2-i}\alpha^{\!\!*}  (-1)^i\,t^{2i+1}\cr
&+\sum_i\binom\alpha i\binom{\frac{3\alpha-1}2-i+\frac{q+1}2}\alpha^{\!\!*} (-1)^i\,t^{2i}\biggr].
\end{split}
\end{equation}
This completes the proof of Lemma~\ref{L3.1}.
\end{proof}

To prove Theorem~\ref{T1.1}, we will use the following criterion \cite[Lemma~7.3]{LN}: A function $g:\Bbb F_q\to\Bbb F_q$ is a permutation of $\Bbb F_q$ if and only if 
\[
\sum_{x\in\Bbb F_q}g(x)^s=
\begin{cases}
0&\text{if}\ 1\le s\le q-2,\cr
-1&\text{if}\ s=q-1.
\end{cases}
\]

\begin{proof}[Proof of Theorem~\ref{T1.1}]
($\Rightarrow$) $1^\circ$ We first show that $q$ must be odd. Otherwise, let $\alpha=1$ and $\beta=q-2$ in \eqref{3.1}, and note that the second sum at the right side does not occur since $\frac{3-1}2-i+\frac{q+1}2\notin\Bbb Z$. We have 
\[
 \sum_{x\in\Bbb F_{q^2}^*}f(x)^{1+(q-2)q}=\sum_i\binom 1i\binom{\frac{3-1}2-i}1 t^{2i}=1.
\]
Note that $f(0)=0$ since $q>2$. Thus we have $\sum_{x\in\Bbb F_{q^2}}f(x)^{1+(q-2)q}= 1\ne 0$,
which is a contradiction.

$2^\circ$ Again, let $\alpha=1$ and $\beta=q-2$ in \eqref{3.1}. We have
\[
\begin{split}
0\,&=(-1)^{\frac{q+1}2}\eta(t)\sum_i\binom 1i\binom{1-i}1(-1)^it^{2i+1}+\sum_i\binom 1i\binom{1-i+\frac{q+1}2}1(-1)^it^{2i}\cr
&=(-1)^{\frac{q+1}2}\eta(t)t+\frac 32-\frac12t^2,
\end{split}
\]
where $\eta$ is the quadratic character of $\Bbb F_q$. Let $\epsilon=(-1)^{\frac{q+1}2}\eta(t)=\pm 1$. Then the above equation becomes
\[
t^2-2\epsilon t-3=0,
\]
i.e.,
\[
(t+\epsilon)(t-3\epsilon)=0.
\]
Thus $t=-\epsilon$ or $3\epsilon$.

First assume $t=-\epsilon$. Then 
\[
\epsilon=(-1)^{\frac{q+1}2}\eta(-\epsilon)=(-1)^{\frac{q+1}2}(-\epsilon)^{\frac{q-1}2}=-\epsilon^{\frac{q-1}2}.
\]
So $\epsilon=-1$ and $q\equiv 1\pmod 4$. This is case (i).

Next assume $t=3\epsilon$. Then
\[
\epsilon=(-1)^{\frac{q+1}2}\eta(3\epsilon)=(-1)^{\frac{q+1}2}\epsilon^{\frac{q-1}2}\eta(3),
\]
i.e.,
\[
\epsilon^{\frac{q+1}2}=(-1)^{\frac{q+1}2}\eta(3).
\]
If $\epsilon=-1$, we have $\eta(3)=1$, which happens if and only if $q\equiv\pm1\pmod{12}$ \cite[\S5.2]{Ire-Ros}. This is case (ii). If $\epsilon=1$, we have $(-1)^{\frac{q+1}2}\eta(3)=1$. There are two possibilities: $\eta(3)=1$ and $(-1)^{\frac{q+1}2}=1$, or $\eta(3)=-1$ and $(-1)^{\frac{q+1}2}=-1$. The first possibility occurs if and only if $q\equiv\pm 1\pmod{12}$ and $q\equiv -1\pmod 4$, i.e., $q\equiv -1\pmod{12}$. The second possibility occurs if and only if $q\equiv\pm 5\pmod{12}$ and $q\equiv 1\pmod 4$, i.e., $q\equiv 5\pmod{12}$. Together, we have case (iii).

($\Leftarrow$) $1^\circ$ We first show that $0$ is the only root of $f$ in $\Bbb F_{q^2}$. Assume to the contrary that there exists $x\in\Bbb F_{q^2}^*$ such that $f(x)=0$. Then we have $x^{2q-2}=-t$, hence $(-t)^{\frac{q+1}2}=1$. However, this cannot be true in any of the cases (i) -- (iii).

$2^\circ$. By $1^\circ$, we have $\sum_{x\in\Bbb F_{q^2}}f(x)^{q^2-1}=-1$. Therefore, it remains to prove that 
\[
\sum_{x\in\Bbb F_{q^2}^*}f(x)^s=0\qquad\text{for}\ 1\le s\le q^2-2.
\]
Write $s=\alpha+\beta q$, where $0\le\alpha,\beta\le q-1$. By Lemma~\ref{L3.1}, it suffices to assume that $\alpha+\beta=q-1$ and $\alpha$ is odd. 

First we consider case (i) in Theorem~\ref{T1.1}. By the second line of \eqref{3.2}, we have 
\[
\begin{split}
\sum_{x\in\Bbb F_{q^2}}f(x)^{\alpha+\beta q}\,&=-\sum_{i+j=\frac{\alpha-1}2}\binom\alpha i\binom{q-1-\alpha}j-\sum_{i+j=\frac{\alpha+q}2}\binom\alpha i\binom{q-1-\alpha}j\cr
&=-\binom{q-1}{\frac{\alpha-1}2}-\binom{q-1}{\frac{\alpha+q}2}\cr
&=-(-1)^{\frac{\alpha-1}2}-(-1)^{\frac{\alpha+q}2}\cr
&=0\kern3.2cm \text{(since $q\equiv 1\kern-1mm\pmod 4$)}.
\end{split}
\]

Now we consider cases (ii) and (iii) of Theorem~\ref{T1.1}. By Lemma~\ref{L3.1}, it suffices to show that for each odd integer $\alpha$ with $0<\alpha<q-1$ we have
\[
(-1)^{\frac{q+1}2}\eta(t)\sum_i\binom\alpha i\binom{\frac{3\alpha-1}2-i}\alpha(-1)^it^{2i+1}
+\sum_i\binom\alpha i\binom{\frac{3\alpha-1}2-i+\frac{q+1}2}\alpha(-1)^it^{2i} =0.
\]
In case (ii), $t=-3$ and $(-1)^{\frac{q+1}2}\eta(t)=-1$; in case (iii), $t=3$ and $(-1)^{\frac{q+1}2}\eta(t)=1$. So it suffices to show that 
\begin{equation}\label{3.3}
\sum_i\binom\alpha i\binom{\frac{3\alpha-1}2-i}\alpha(-1)^i 3^{2i+1}
+\sum_i\binom\alpha i\binom{\frac{3\alpha-1}2-i+\frac{q+1}2}\alpha(-1)^i 3^{2i} =0.
\end{equation}
Write $\alpha=2n+1$. In $\Bbb Z_p/p\Bbb Z_p$ ($=\Bbb F_p$), the left side of \eqref{3.3} equals 
\[
\begin{split}
\frac1{\alpha!\,2^\alpha}\biggl[&\sum_i\binom{2n+1}i\biggl(\prod_{j=1}^{2n+1}(6n-2i+4-2j)\biggr)(-1)^i3^{2i+1}\cr
+\,&\sum_i\binom{2n+1}i\biggl(\prod_{j=1}^{2n+1}(6n-2i+5-2j)\biggr)(-1)^i3^{2i}\biggr].
\end{split}
\]
(Lemma~\ref{L3.0} is used to obtained the second sum in the above.)
By Theorem~\ref{T1.2}, the above expression is $0$, and we are done.
\end{proof}

\noindent{\bf Note.} For a different proof for the sufficiency of case (i) in Theorem~\ref{T1.1}, see \cite[Theorem~5.9]{FHL}. 


\section{The Polynomial $g_{n,q}$}

In this section, we briefly discuss the connection of Theorem~\ref{T1.1} to a recent study of a class of PPs defined by a functional equation.

Let $q$ be a prime power and $n\ge 0$ an integer. The functional equation 
\[
\sum_{a\in\Bbb F_q}({\tt x}+a)^n=g_{n,q}({\tt x}^q-{\tt x})
\]
defines a polynomial $g_{n,q}\in\Bbb F_p[{\tt x}]$, where $p=\text{char}\,\Bbb F_q$. Frequently, $g_{n,q}$ is a PP of $\Bbb F_{q^e}$; when this happens, the triple $(n,e;q)$ is called {\em desirable}. Many new and interesting PPs in the form of $g_{n,q}$ have been found \cite{FHL, Hou11,Hou12,HMSY}. Assume $q>2$ and $i>0$. It is known \cite[Theorem~5.9]{FHL} that
\[
g_{q^{2i}-q-1,q}({\tt x})\equiv(i-1){\tt x}^{q^2-q-1}-i{\tt x}^{q-2}\pmod{{\tt x}^{q^2}-{\tt x}}.
\]
Therefore we have the following restatement of Theorem~\ref{T1.1}:

\begin{thm}\label{T4.1}
Assume $q>2$, $i>0$, and $i\ne 0,1\pmod p$. Then $(q^{2i}-q-1,2;q)$ is desirable if and only if one of the following holds:
\begin{itemize}
  \item [(i)] $2i\equiv 1\pmod p$ and $q\equiv 1\pmod 4$;
  \item [(ii)] $2i\equiv-1\pmod p$ and $q\equiv\pm 1\pmod{12}$;
  \item [(iii)] $4i\equiv1\pmod p$ and $q\equiv-1\pmod 6$.
\end{itemize} 
\end{thm}


\end{document}